\documentclass[12pt, a4paper]{amsart}

\usepackage[margin=1.3in]{geometry}
\usepackage[ansinew]{inputenc}
\usepackage[all]{xy}
\SelectTips{cm}{}
\usepackage[
  ,pdfborder={0 0 0}
  ,urlcolor=black,a4paper,hypertexnames=false]{hyperref}
\hypersetup{pdfauthor={pietro_capovilla}
  ,pdftitle=small_aspherical_fillings}

\usepackage{amsfonts}
\usepackage{amssymb, amsmath,amsthm, mathtools}
\usepackage{tabularx, hyperref}
\usepackage{enumerate, comment}

\usepackage{tikz-cd}

\usepackage[textsize=tiny]{todonotes}
\setuptodonotes{fancyline, color=blue!30}

\makeatletter
\newtheorem*{rep@theorem}{\rep@title}
\newcommand{\newreptheorem}[2]{%
	\newenvironment{rep#1}[1]{%
		\def\rep@title{#2 \ref{##1}}%
		\begin{rep@theorem}}%
		{\end{rep@theorem}}}
\makeatother

\newtheorem{intro_thm}{Theorem}
\newtheorem{intro_cor}[intro_thm]{Corollary}

\newtheorem{intro_quest}{Question}
\newtheorem{intro_conj}{Conjecture}

\newtheorem{lemma}{Lemma}[section]

\newreptheorem{thm}{Theorem}  

\newreptheorem{prop}{Proposition}  
\newtheorem{cor}[lemma]{Corollary}
\newreptheorem{cor}{Corollary}  

\theoremstyle{definition}

\newreptheorem{quest}{Question}  

\newtheorem{rem}[lemma]{Remark}

\theoremstyle{definition}

\renewcommand{\epsilon}{\varepsilon}
\renewcommand{\theta}{\vartheta}

\renewcommand{\phi}{\varphi}

\newcommand{\twprod}{\mathbin{%
		\ooalign{\raise1.15ex\hbox{$\scriptstyle\sim$}\cr\hidewidth$\times$\hidewidth\cr}%
}}
\newcommand{\twprodboth}{\mathbin{%
		\ooalign{\raise1.15ex\hbox{$\scriptstyle(\sim)$}\cr\hidewidth$\times$\hidewidth\cr}%
}}

\DeclareMathAlphabet{\mathcal}{OMS}{cmsy}{m}{n}

\newcommand{\Z}{\ensuremath{\mathbb{Z}}}
\newcommand{\R}{\ensuremath{\mathbb{R}}}
\newcommand{\N}{\ensuremath{\mathbb{N}}}

\DeclareMathOperator{\MCG}{MCG}
\DeclareMathOperator{\SL}{SL}
\DeclareMathOperator{\ab}{ab}
\DeclareMathOperator{\id}{id}
\DeclareMathOperator{\Spec}{Spec}

\long\def\forget#1{}

\makeatletter
\@namedef{subjclassname@2020}{%
  \textup{2020} Mathematics Subject Classification}
\makeatother

\begin{document}

\title[Aspherical 4-manifolds and their geography]{Aspherical 4-manifolds with positive Euler characteristic and their geography}

\author[]{Pietro Capovilla}
\address{Scuola Normale Superiore, Pisa Italy}
\email{pietro.capovilla@sns.it}
\curraddr{University of Cyprus, Nicosia, Cyprus}
\email{capovilla.pietro@ucy.ac.cy}

\thanks{}

\subjclass[1000]{}
\begin{abstract}
	We present an explicit construction of closed oriented aspherical smooth 4-manifolds with $\chi = \sigma = n$ for every positive integer $n$.
	This proves a conjecture of Edmonds by providing a closed oriented aspherical 4-manifold with Euler characteristic 1, and it shows that the real analogue of the Bogomolov–Miyaoka–Yau inequality fails for aspherical 4-manifolds.
	By the Hitchin-Thorpe inequality, these manifolds do not admit Einstein metrics.
	As a further consequence of our construction, we show that every closed aspherical 3-manifold with amenable fundamental group is virtually the $\pi_1$-injective boundary of an aspherical 4-manifold with vanishing Euler characteristic and vanishing simplicial volume, thereby answering questions of Edmonds and of L{\"o}h–Moraschini–Raptis up to finite covers.
\end{abstract}

\maketitle

\section{Introduction}

The geometry and the topology of closed aspherical manifolds is known to be highly constrained.
A classical conjecture, attributed to Singer, predicts that the $L^2$-Betti numbers of aspherical manifolds vanish away from the middle dimension.
By a result of L{\"u}ck \cite[Theorem~5.1]{Lue94}, this general conjecture implies, in dimension 4, the following conjecture on the geography of aspherical manifolds.
\begin{intro_conj}
	\label{conj:geography}
	If $X$ is a closed oriented aspherical 4-manifold, then
	\begin{equation}
		\label{eq:geography}
		\chi(X)\geq |\sigma(X)|,
	\end{equation}
	where $\chi(X)$ and $\sigma(X)$ denote the Euler characteristic and the signature of $X$.
\end{intro_conj}
For further background on Conjecture~\ref{conj:geography}, we refer the reader to \cite{Lue94, Lue02} and the references therein, and to \cite{Edm20, ADCL25} for more recent discussions.

Conjecture~\ref{conj:geography} implies, in particular, another classical conjecture attributed to Hopf and Thurston, which predicts that the Euler characteristic of aspherical 4-manifolds is non-negative.
We show that (\ref{eq:geography}) is sharp for every non-negative Euler characteristic.
\begin{intro_thm}
	\label{thm:aspherical_manif}
	For every natural number $n \in \N$, there is a closed oriented connected aspherical smooth 4-manifold $X_n$ such that $\chi(X_n)=\sigma(X_n)=n$.
\end{intro_thm}
The existence of $X_1$ was conjectured by Edmonds \cite[Conjecture~1]{Edm20}.
Moreover, the manifolds $X_1$, $X_5$, $X_7$, and $X_{11}$ appear to be the first known examples of closed oriented aspherical 4-manifolds with Euler characteristic 1, 5, 7, and 11, respectively \cite{Edm20}.

Theorem~\ref{thm:aspherical_manif} further shows that the real analogue of the Bogomolov–Miyaoka–Yau inequality -- namely, the inequality $\chi \geq 3 |\sigma|$, which is known to hold, for instance, for compact complex surfaces of general type \cite{Bog78, Miy77, Yau77}, for surface bundles over surfaces \cite{Kot98, Ham} and for some geometrically decomposable manifolds \cite{DCG} -- does not hold for oriented closed aspherical 4-manifolds. 
In particular, Theorem~\ref{thm:aspherical_manif} disproves \cite[Conjecture~4]{Edm20}.
On the other hand, the manifolds $X_n$ satisfy the Singer conjecture (see Remark~\ref{rem:singer_conjecture}).

The Hitchin-Thorpe inequality provides a topological obstruction to the existence of Einstein metrics on smooth 4-manifold, stating that $2\chi \geq 3 |\sigma|$ for every closed oriented 4-manifold admitting such a metric \cite{Hit74}. 
It follows that the manifolds $X_n$ do not admit Einstein metrics for any $n\geq 1$.
This yields the first aspherical examples with non-vanishing signature (see \cite{DCG} and the references therein).\\

The existence of $X_1$ allows us to realize aspherical manifolds with every admissible Euler characteristic also in higher dimensions.
Let $\Spec_\chi(m)\subseteq \Z$ denote the spectrum of the possible Euler characteristics of closed orientable aspherical $m$-manifolds.
By Poincaré duality, we have $\Spec_\chi(m)={0}$ whenever $m$ is odd.
A further consequence of Poincaré duality is that $\Spec_\chi(m)\subseteq 2\cdot \Z$ whenever $m\equiv 2 \pmod{4}$.
Regarding aspherical manifolds specifically, the Hopf–Thurston conjecture predicts that $\Spec_\chi(m)\subseteq (-1)^k\cdot \N$ when $m = 2k$.

Thanks to Theorem \ref{thm:aspherical_manif}, we can realize every Euler characteristic compatible with these obstructions by taking suitable products.
Define $\omega(m) \in \Z$ as follows: set $\omega(m)=1$ when $m\equiv0\pmod4$, $\omega(m)=-2$ when $m\equiv2\pmod4$, and $\omega(m)=0$ otherwise.
\begin{intro_cor}
	\label{cor:asph_every_dim}
	For every natural number $m \geq 1$, we have $\omega(m)\cdot \N \subseteq \Spec_\chi(m)$.
\end{intro_cor}
\begin{proof}
	If $m$ is odd, then $\omega(m)=0$ and the $m$-torus realizes $\chi=0$.
	If $m \equiv 0\pmod{4}$, then every non-negative integer $n$ can be realized as the Euler characteristic of products $X_1\times \cdots\times X_1\times X_n$.
	If $m\equiv 2\pmod{4}$, then every non-positive even integer can be realized as the Euler characteristic of products $X_1\times \cdots\times X_1\times \Sigma_g$, where $\Sigma_g$ is a closed surface of genus $g\geq 1$.
\end{proof}

The strategy underlying the construction of the manifolds in Theorem~\ref{thm:aspherical_manif} is to glue together some building blocks along their boundary components.
Let $M$ be an oriented closed aspherical 3-manifold, not necessarily connected. 
A null-cobordism $W$ of $M$ is called an \emph{aspherical filling} of $M$ if it is an oriented compact connected aspherical 4-manifold with $\pi_1$-injective boundary $\partial W = M$.

Since gluing aspherical spaces along $\pi_1$-injective aspherical subspaces preserves asphericity \cite{Whi78} \cite[Theorem 3.1]{Edm20}, the union of two aspherical fillings along diffeomorphic boundary components is again an aspherical filling.
Moreover, the Euler characteristic satisfies the sum formula
\[
\chi(A\cup B)=\chi(A) + \chi(B)-\chi(A\cap B),
\]
thus it is additive when gluing together 4-manifolds along their boundary components.
Similarly, Novikov additivity implies that the signature of oriented 4-manifolds is additive in presence of (orientation-reversing) gluings along boundary components \cite[Theorem~5.3]{Kir89}.
These additivity properties will be used repeatedly in our proof of Theorem~\ref{thm:aspherical_manif}.
All constructions in this work are carried out in the smooth category, and all manifolds are implicitly assumed to be smooth.

The building blocks used in the construction of the manifolds in Theorem~\ref{thm:aspherical_manif} are of three types.
The contribution to the Euler characteristic comes from certain minimal-volume complex hyperbolic manifolds with cusps, classified by Di Cerbo and Stover \cite{DCS18} (see Lemma~\ref{lem:DiCerboStover}).
To cap off these manifolds, we use torus bundles over punctured spheres (Lemma~\ref{lem:prod}), together with a new building block whose boundary is a semi-bundle (Lemma~\ref{lem:double_covering_trick}).
Both of the latter types of pieces have vanishing Euler characteristic.\\

It is well known that every closed oriented 3-manifold bounds some oriented 4-manifold. 
By relative hyperbolization \cite{DJ91, DJW01}, every closed aspherical 3-manifold admits an aspherical filling.
A natural question concerns the minimal complexity of an aspherical filling of a given aspherical $3$-manifold.
This type of isoperimetric problem admits several formulations, depending on the chosen notion of complexity.
In both Edmonds’s work \cite{Edm20} and the present one, the Euler characteristic plays a central role.
\begin{intro_quest}[{\cite{Edm20}}]
	\label{quest:euler_char}
	Does every closed oriented aspherical 3-manifold admit an aspherical filling with vanishing Euler characteristic?
\end{intro_quest}

Another relevant invariant is Gromov’s \emph{simplicial volume} \cite{Gro82}.
The relationship between the simplicial volume and the Euler characteristic for aspherical manifolds is poorly understood.
In particular, Gromov asked whether the vanishing of the simplicial volume of a closed aspherical manifold implies the vanishing of its Euler characteristic \cite[p.~232]{Gro93}.
This conjecture remains open, although recent work by L{\"o}h, Moraschini and Raptis \cite{LMR22} explores new approaches to the problem.
In particular, exploiting the additivity properties of both invariants, they formulate an equivalent version of Gromov’s conjecture for manifolds with boundary.
\begin{intro_quest}[{\cite[Question~2.13]{LMR22}}]
	\label{que:rel_gromov_conj}
	Let $W$ be an oriented compact aspherical 4-manifold with $\pi_1$-injective aspherical boundary. 
	Does the vanishing of the simplicial volume of $W$ imply the vanishing of its Euler characteristic?
\end{intro_quest}
This extension to manifolds with boundary provides a systematic framework to compare the filling properties of the simplicial volume and the Euler characteristic.
In this setting, L{\"o}h, Moraschini and Raptis also posed the following question.
\begin{intro_quest}[{\cite[Question~3.28]{LMR22}}]
	\label{quest:sv}
	Does every closed oriented aspherical 3-manifold with amenable fundamental group admit an aspherical filling with vanishing simplicial volume?
\end{intro_quest}
An affirmative answer to both Question~\ref{que:rel_gromov_conj} and Question~\ref{quest:sv} would in turn yield an affirmative answer to Question~\ref{quest:euler_char}, at least for 3-manifolds with amenable fundamental group.
The possible interplay between Question~\ref{que:rel_gromov_conj} and Question~\ref{quest:sv} in the construction of closed aspherical manifolds with small odd Euler characteristic is also highlighted in \cite[Proposition~3.31]{LMR22}.
On the other hand, the manifolds $X_n$ from Theorem~\ref{thm:aspherical_manif} have positive simplicial volume (see Remark \ref{rem:sv_examples}).\\

In this work, we answer both Question~\ref{quest:euler_char} and Question~\ref{quest:sv} for certain families of torus bundles and semi-bundles. 
As a further consequence, we show that both statements hold \emph{virtually}, that is, up to finite covers, for all manifolds with amenable fundamental group.
\begin{intro_thm}
	\label{thm:virtual_small_aspherical_fillings}
	Every oriented closed connected aspherical 3-manifold with amenable fundamental group virtually admits an aspherical filling with vanishing Euler characteristic and vanishing simplicial volume.
\end{intro_thm}
In Section~\ref{sec:torus_bundle_boundaries}, we describe the building blocks of our construction which are already available in the literature.
The main contribution of this work is the construction of aspherical fillings with semi-bundle boundaries, which is presented in Section~\ref{sec:semi_bundle_boundaries}, together with the proof of Theorem~\ref{thm:aspherical_manif}.
The proof of Theorem~\ref{thm:virtual_small_aspherical_fillings} can be found in Section~\ref{sec:computing_simplicial_volume}.

\subsection*{Acknowledgments}
I am grateful to my PhD advisor Roberto Frigerio for his guidance and interest in this project.
I am deeply indebted to Stefano Riolo, who pointed out the work of Di Cerbo and Stover \cite{DCS18} after reading a previous version of this manuscript.
I also wish to thank Marco Moraschini for suggesting Corollary~\ref{cor:asph_every_dim}, and Ervin Had\v{z}iosmanovi\'c, Federica Bertolotti and Elena Bogliolo for their helpful comments.
Finally, this work is supported by the start-up funding of Prof. Neofytidis, and has been written within the activities of GNSAGA group of INdAM.

\section{Filling torus bundles}
\label{sec:torus_bundle_boundaries}
The goal of this section is to describe a collection of building blocks that will play a crucial role in the proof of Theorem~\ref{thm:aspherical_manif} and Theorem~\ref{thm:virtual_small_aspherical_fillings}.

Given an oriented manifold $M$, we denote by $-M$ a copy of $M$ with the opposite orientation.
We observe that $M$ has an aspherical filling with Euler characteristic $n \in \Z$ if and only if $-M$ does, since the Euler characteristic of an orientable manifold is independent of its orientation.

We denote by $I$ the interval $[0,1]$, and by $T$ the 2-dimensional torus.
Fix a basis of $\pi_1(T)$ so that the (positive) mapping class group $\MCG(T)$ of the torus can be identified with the group $\SL_2(\Z)$.
We denote by $\id$ the identity element of $\SL_2(\Z)$.

For $\phi \in \SL_2(\Z)$, we denote by $T(\phi)$ the \emph{torus bundle} with monodromy $\phi$, i.e., the 3-manifold fibering over the circle with torus fiber obtained from the product $T \times [0,1]$ by identifying $(x,0)$ with $(f(x),1)$, where $f$ is any diffeomorphism representing $\phi$.
It is well known that this construction does not depend on the choice of $f$ and produces an aspherical manifold.
Torus bundles are classified by the conjugacy class of their monodromy in $\MCG(T) = \SL_2(\Z)$ \cite[Theorem 2.6]{Hat}.
If $T$ and $[0,1]$ are oriented, then $T \times [0,1]$ carries the product orientation, which induces an orientation on $T(\phi)$, since $\phi$ is represented by an orientation-preserving diffeomorphism.
In this situation, we have $T(\phi^{-1}) = -T(\phi)$.

\begin{rem}
	\label{rem:orientations}
	We make explicit our conventions regarding orientations.
	If $W$ is an oriented manifold with boundary $\partial W$, the boundary $\partial W$ is oriented according to the outward-normal first convention, see \cite[Definition 1.1.2]{GS99}.
	Let $p\colon E \to B$ be a smooth fiber bundle whose base space $B$ and fiber $F$ are oriented manifolds with volume forms $\omega_B$ and $\omega_F$, respectively. 
	The total space $E$ is then oriented by the volume form $\omega_E = \omega_F \wedge p^*(\omega_B)$.
\end{rem}

We begin by considering torus bundles over punctured disks or spheres.
We orient the punctured disk as an oriented submanifold of $\mathbb{R}^2$ so that the outer boundary circle is oriented clockwise, while each inner boundary circle is oriented counterclockwise.
A torus bundle over such a punctured disk restricts to a torus bundle over each boundary component. 
The bundle is completely determined by the monodromies around the punctures: the monodromy around the outer circle is equal to the inverse of the composition of the monodromies around the inner circles, taken in the appropriate order.
\begin{lemma}[{\cite[Lemma~6.1]{Edm20}}]
	\label{lem:prod}
	Let $\phi_1,\dots, \phi_n\in \SL_2(\Z)$ be such that $\phi_n \cdots \phi_1 = \textup{id}$. Then there is a bundle over the $n$-punctured 2-sphere with torus fiber and monodromies around the punctures given by $\phi_1, \dots, \phi_n$.
	
	In particular, there is a compact oriented connected aspherical 4-manifold $W$ with $\chi(W)=0$ and $\pi_1$-injective aspherical boundary such that
	\[
	\partial W = T(\phi_1)\sqcup \dots \sqcup T(\phi_n).
	\]
\end{lemma}
The signature of the torus bundle from Lemma~\ref{lem:prod} can be computed through Meyer signature formula \cite[Satz~5]{Mey73}: $\sigma(W)=\mathcal{M}(\phi_1)+ \dots + \mathcal{M}(\phi_n)$, where $\mathcal{M}\colon \SL_2(\Z)\rightarrow \tfrac{1}{3}\Z$ is the Meyer function of genus one, as described in \cite[Section~5]{Mey73}.
We remark that, compared to \cite{Mey73}, our formula differs by a sign, which arises from the different convention used here to orient $W$.

Let $\SL_2(\Z)'$ denote the derived subgroup of $\SL_2(\Z)$. 
The following observation is well known and can be found, for instance, in \cite[Section~5.6]{Foz} or \cite[Claim~1]{Moz04}. For the reader’s convenience, we include a proof here.
\begin{lemma}
	\label{lem:commutator}
	Let $\phi \in \SL_2(\Z)'$. 
	There is a compact oriented connected aspherical 4-manifold $W$ with $\chi(W)=0$ and $\pi_1$-injective aspherical boundary such that $\partial W = T(\phi)$.
\end{lemma}
\begin{proof}
	Since $\varphi \in \SL_2(\Z)'$ can be written as a product of commutators, it suffices to construct aspherical fillings for these commutators and glue them together using the bundle from Lemma~\ref{lem:prod}.
	Asphericity and vanishing of the Euler characteristic are preserved, because the gluings take place along $\pi_1$-injective aspherical boundary components with vanishing Euler characteristic.
	Hence, it is enough to prove that every commutator in $\SL_2(\Z)$ admits an aspherical filling with vanishing Euler characteristic.
	
	Let $\varphi = [g,h]$ for some $g,h \in \SL_2(\Z)$. 
	It follows that $\varphi h$ is conjugated to $h$ in $\SL_2(\Z)$, hence $T(\phi h)$ and $T(h)$ are diffeomorphic \cite[Theorem 2.6]{Hat}. 
	By Lemma~\ref{lem:prod}, there is a torus bundle $V$ over a pair of pants such that
	\[
	\partial V = T(\phi h) \sqcup T(h^{-1}) \sqcup T(\phi^{-1}).
	\]
	Gluing the diffeomorphic boundary components of $V$ yields an aspherical filling $W$ of $T(\varphi^{-1})$, and hence of $T(\varphi) = -T(\varphi^{-1})$.
	By construction, $W$ is the total space of a torus bundle over a surface with boundary. In particular, it has vanishing Euler characteristic.
\end{proof}
\begin{cor}
	\label{cor:modulo_derived_subgroup}
	If $\phi, \psi \in \SL_2(\Z)$ are congruent modulo $\SL_2(\Z)'$, then $T(\phi)$ has an aspherical filling with Euler characteristic $n \in \N$ if and only if $T(\psi)$ does.
\end{cor}
\begin{proof}
	By assumption, we can write $\phi = \psi h$ for some $h \in \SL_2(\Z)'$. Lemma~\ref{lem:prod} then provides an oriented aspherical manifold $W$ with $\chi(W)=0$ and $\pi_1$-injective boundary
	\[
	\partial W = T(\psi)\sqcup T(h) \sqcup -T(\phi).
	\]
	Since $h$ belongs to the derived subgroup of $\SL_2(\Z)$, Lemma~\ref{lem:commutator} ensures that the boundary component $T(h)$ can be capped off by an aspherical filling without changing the Euler characteristic. 
	The claim then follows.
\end{proof}

Lemma~\ref{lem:commutator} shows that torus bundles with monodromy in $\SL_2(\Z)'$ provide examples for which Question~\ref{quest:euler_char} can be answered affirmatively, i.e., they admit aspherical fillings with vanishing Euler characteristic. 
Let
\[
B = \begin{pmatrix} 1 & 1 \\ 0 & 1 \end{pmatrix}
\]
be the element of $\SL_2(\Z)$ representing a positive Dehn twist about a non-separating simple closed curve.
While this representative is not unique (all such twists are conjugate), we will use this specific matrix throughout.
It is well known that the abelianization $\SL_2(\mathbb{Z})^{\ab}$ of $\SL_2(\mathbb{Z})$ is a finite cyclic group of order 12 generated by the class of $B$ (see, for instance, \cite[Theorem VIII.2]{New72}).
Therefore, by Corollary~\ref{cor:modulo_derived_subgroup}, verifying Question~\ref{quest:euler_char} for all torus bundles reduces to considering powers of $B$.
In the next section, we prove that Question~\ref{quest:euler_char} can be answered affirmatively for torus bundles whose monodromy is an even power of $B$ (see Corollary~\ref{cor:even_powers_are_fillable}), which is the key step in our proof of Theorem~\ref{thm:aspherical_manif}.

The only building block that contributes to the Euler characteristic in the manifolds appearing in Theorem~\ref{thm:aspherical_manif} is the compact manifold obtained by truncating the cusps of one of the minimal-volume smooth complex hyperbolic surfaces admitting a smooth toroidal compactification, as classified by Di Cerbo and Stover in \cite{DCS18}.
\begin{lemma}[{\cite{DCS18}}]
	\label{lem:DiCerboStover}
	There is a compact oriented connected aspherical 4-manifold $W$ with $\chi(W)=\sigma(W)=1$ and $\pi_1$-injective aspherical boundary
	\[
	\partial W = T(B)\sqcup T(B^3).
	\]
\end{lemma}
\begin{proof}
	Let $\mathbb{H}^2_\mathbb{C}$ denote the complex hyperbolic plane (of real dimension 4) and let $\Gamma$ denote a torsion free non-uniform lattice of isometries of $\mathbb{H}^2_\mathbb{C}$.
	Di Cerbo and Stover classify in \cite{DCS18} those complex hyperbolic surfaces $\mathbb{H}^2_\mathbb{C}/\Gamma$ that admit a smooth toroidal compactification $(X,D)$ satisfying
	\[
	3\bar{c}_2(X,D)=\bar{c}_1^2(X,D)=3,
	\]
	where $\bar{c}_1$ and $\bar{c}_2$ denote the logarithmic Chern numbers of the pair $(X,D)$. 
	Here $X$ is a smooth projective variety (hence a closed manifold of real dimension 4) and $D$ is a boundary divisor consisting of the union of smooth disjoint elliptic curves (hence tori of real dimension 2) such that $X\setminus D$ is biholomorphic to $\mathbb{H}^2_\mathbb{C}/\Gamma$ (see also \cite{HS96}).
	Their classification yields exactly five such toroidal compactifications: the first is obtained by blowing up an abelian surface at a point, while the remaining four arise from bielliptic surfaces, each blown up once \cite[Theorem 1.1]{DCS18}.
	
	Let	$(X,D)$ be one of these latter four compactifications. For our purposes, the specific choice among them is irrelevant.
	We define $W$ to be the compact core of the cusped manifold $X\setminus D = \mathbb{H}^2_\mathbb{C}/\Gamma$, i.e., the compact manifold obtained from $X\setminus D$ by truncating the cusps, with the orientation induced from the complex structure. 
	It is immediate that $W$ is aspherical with $\pi_1$-injective aspherical boundary.
	We now verify that $W$ also satisfies the remaining properties in the statement.
	
	The divisor $D$ consists of two disjoint elliptic curves $D_1$ and $D_2$ with self-intersection numbers $D_1^2=-1$ and $D_2^2=-3$.
	A regular neighborhood of $D_i$ in $X$ is a disk bundle over a torus, whose boundary is therefore a circle bundle with Euler number $D_i^2$ (see, for example, \cite[Exercise~1.4.11(b)]{GS99}).
	It follows that $W$ has two boundary components, consisting of circle bundles over a torus with Euler number 1 and 3, respectively.
	Since every circle bundle over a torus with Euler number $k\in \Z$ is diffeomorphic to the torus bundle over the circle with monodromy $B^k$ (see, for instance, \cite[Exercise~11.4.11]{Mar}), it follows that $W$ has the claimed boundary components.
	This can be also read from the presentation of the parabolic subgroups of $\Gamma$ in \cite[Section~7.1]{DCS18}.
	
	Let $c_1(X)$ and $c_2(X)$ denote the ordinary Chern numbers of $X$.
	Since $D$ is a union of elliptic curves, we have that
	\[
	\chi(W)=\chi(X\setminus D)=\chi(X)-\chi(D) = \chi(X) = c_2(X) = \bar{c}_2(X)
	\]
	(see, for example, \cite[p. 5]{DCS18}).
	From $3\bar{c}_2(X,D)=\bar{c}_1^2(X,D)=3$, it follows that $\chi(W)=1$.
	
	The signature of $W$ can be computed as in \cite[Section~5]{DCG}. 
	For the convenience of the reader, we briefly include the argument.
	We first compute the signature of $X$. 
	By the adjunction formula, we have that $\bar{c}_1^2(X,D) = c_1^2(X)-D_1^2 - D_2^2$ (see, for example, \cite[p. 6]{DCS18}), which implies $c_1^2(X)=-1$.
	The Hirzebruch signature formula then gives
	\[
	\sigma(X)=\frac{1}{3}(c_1^2(X)-2c_2(X))=-1.
	\]
	
	In order to compute the signature of $W$, we invoke Novikov additivity on the decomposition $X= W \cup N(D_1) \cup N(D_2)$, where $N(D_i)$ denotes the regular neighborhood of $D_i$ inside $X$. Since the self-intersection of $D_i$ is negative, we have $\sigma(N(D_i))=-1$. 
	We conclude that
	\[
	\sigma(W)=\sigma(X)-\sigma(N(D_1))-\sigma(N(D_2)) = 1. \qedhere
	\]
\end{proof}
Thanks to Lemma \ref{lem:DiCerboStover}, in order to obtain a closed oriented aspherical 4-manifold with Euler characteristic 1 it is enough to show that $T(B)\sqcup T(B^3)$ admits an aspherical filling with vanishing Euler characteristic. By Lemma~\ref{lem:prod}, this reduces to showing that $T(B^{-4})=-T(B^4)$ admits an aspherical filling with vanishing Euler characteristic.

\section{Filling semi-bundles}
\label{sec:semi_bundle_boundaries}
In this section we construct additional building blocks for our proof of Theorem~\ref{thm:aspherical_manif}. They are characterized by having semi-bundle boundaries and represent the main contribution of this work.

Let $K$ denote a Klein bottle.
There exists a unique interval bundle $N = K\twprod I$ over $K$ with orientable total space.
Its boundary is the orientable double cover of $K$, namely a torus, which is $\pi_1$-injective in $N$. 
Equivalently, $N$ can be described as the quotient of $S^1\times S^1\times [0,1]$ by the involution
\[
\sigma(x,y,z)=(x+\pi, -y, 1-z),
\]
where the circle $S^1$ is identified with $\R/2\pi\Z$ \cite[Section 2.2]{Hat}.
We endow $S^1$ and $[0,1]$ with the standard orientation induced from $\mathbb{R}$, and $S^1 \times S^1 \times [0,1]$ with the corresponding product orientation.
Since $\sigma$ is orientation-preserving, $N$ inherits the orientation of $S^1 \times S^1 \times [0,1]$, which we fix throughout this section.

Let $C_x$ and $C_y$ denote two circles in $S^1 \times S^1 \times \{0\}$ defined by fixing the coordinate $y$ and $x$, respectively.
Their images in $N$ define horizontal and vertical directions in the boundary torus $\partial N = T$.
Choosing orientations for $C_x$ and $C_y$ determines a basis for the fundamental group of $\partial N = T$, which we fix once and for all.
With respect to this basis, the positive mapping class group $\MCG(T)$ can be identified with $\SL_2(\Z)$.

Given $\phi \in \SL_2(\Z)$, the \emph{semi-bundle} associated with $\phi$ is the manifold 
\[
N(\phi) = N \cup_\phi -N,
\]
obtained by gluing two copies of $N$ with opposite orientations along their boundaries via (a representative of) $\phi$.
Since $\phi$ is represented by orientation-preserving diffeomorphisms and $-N$ carries the opposite orientation of $N$, the manifold $N(\phi)$ is oriented.
Moreover, $N(\phi)$ is also aspherical, since it is obtained by gluing aspherical manifolds along their $\pi_1$-injective aspherical boundaries.

The manifold $N(\phi)$ is foliated by tori parallel to $\partial N$, together with two Klein bottles lying at the cores of the two copies of $N$ \cite[Section 2.2]{Hat} \cite[Section 11.4.3]{Mar}.
The regular neighborhoods of these Klein bottles are canonically diffeomorphic to $N$, endowed with opposite induced orientations.

We now introduce a self-diffeomorphism of $N$ that will be used later.
Consider the involution
\[
S^1\times S^1 \times [0,1]\rightarrow S^1\times S^1 \times [0,1], \qquad (x,y,z)\mapsto (-x,y,z).
\]
This map commutes with $\sigma$ and therefore induces a self-diffeomorphism
\[
\xi \colon N \rightarrow N
\]
which reverses the orientation \cite[Section 2.2]{Hat}. 
With respect to the chosen coordinates on $\partial N$, the map induced by $\xi$ on $\partial N$ is represented by the matrix
\[
\tau = \begin{pmatrix} -1 & 0 \\ 0 & 1 \end{pmatrix}.
\]
\begin{lemma}
	\label{lem:double_covering_trick}
	Let $\phi \in \SL_2(\Z)$.
	There is a compact oriented connected aspherical 4-manifold $W$ with $\chi(W) = 0$ and $\pi_1$-injective aspherical boundary consisting of
	\[
	\partial W = -N(\phi) \sqcup N(\phi)\sqcup T(\phi \tau \phi^{-1}\tau).
	\]
	In particular, $T(\phi \tau \phi^{-1} \tau)$ has an aspherical filling with $\chi = 0$.
\end{lemma}
\begin{proof}
	We construct $W$ by gluing together two copies of $N(\phi)\times [0,1]$ along the regular neighborhoods of the Klein bottles in $N(\phi)\times \{1\}$ via the map $\xi$.
	
	More precisely, let $W_{-} = N(\phi)\times [0,1]$ and $W_+=N(\phi)\times[2,3]$, where $[0,1]$ carries the standard orientation from $\R$, while $[2,3]$ is endowed with the opposite orientation. 
	Both $W_-$ and $W_+$ are equipped with the product orientation.
	With these choices, we have
	\[
	N(\phi)\times \{0\}=-N(\phi),\qquad N(\phi)\times \{1\} = N(\phi),
	\]
	where the orientation is induced from $W_-$ by the outward-normal-first convention. Similarly,
	\[
	N(\phi)\times \{2\}=N(\phi),\qquad N(\phi)\times \{3\} = -N(\phi).
	\]
	
	Let $K$ and $K'$ be the Klein bottles forming the cores of the two copies of $N$ inside $N(\phi)$, and let $R$ and $R'$ denote their corresponding regular neighborhoods in $N(\phi)$, oriented consistently with $N(\phi)$.
	Both $R$ and $R'$ are canonically diffeomorphic to $N$, but with opposite orientations, say $R=N$ and $R'=-N$.
	Moreover, the complement $N(\phi)\setminus(R\cup R')$ is diffeomorphic to $T\times [0,1]$.

	It follows that $R \times \{1\}$ and $R' \times \{3\}$ can be canonically identified with $N$, while $R' \times \{1\}$ and $R \times \{3\}$ correspond to $-N$.
	The manifold $W$ is obtained from $W_- \sqcup W_+$ by gluing $R \times \{1\} = N$ to $R' \times \{3\} = N$ and $R' \times \{1\} = -N$ to $R \times \{3\} = -N$ via the map $\xi$.
	A schematic representation of $W$ is shown in Figure~\ref{fig:cobordism}.
	
	\begin{figure}[htbp]
		\centering
		\includegraphics[scale=1]{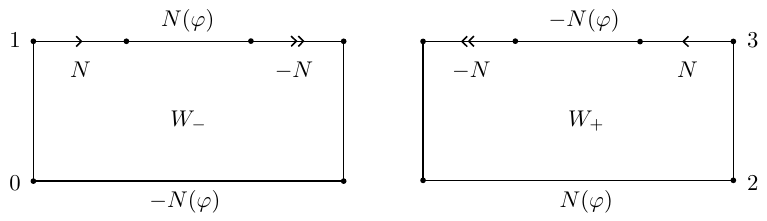}
		\caption{Gluing scheme illustrating Lemma~\ref{lem:double_covering_trick}.}
		\label{fig:cobordism}
	\end{figure}
	
	Since both gluing maps are orientation-reversing, the resulting manifold $W$ is oriented. 
	Moreover, $W$ is aspherical, because $W_-$ and $W_+$ are aspherical, and the gluings are along $\pi_1$-injective aspherical submanifolds.
	
	The boundary of $W$ consists of three components: one copy of $-N(\phi)$ inherited from $W_-$; one copy of $N(\phi)$ inherited from $W_+$; and a third component obtained by gluing together two copies of $N(\phi)\setminus(R\cup R')=T\times [0,1]$. 
	The latter is a torus bundle with monodromy $\phi \tau \phi^{-1}\tau$. 
	All boundary components are $\pi_1$-injective in $W$, and the Euler characteristic of $W$ is
	\[
	\chi(W)=\chi(W_-)+\chi(W_+)-2\chi(N)=0.
	\]
	Gluing the two boundary components $N(\phi)$ and $-N(\phi)$ of $W$ via the identity map then yields an aspherical filling of $T(\phi  \tau  \phi^{-1}  \tau)$ with vanishing Euler characteristic.
\end{proof}
Monodromies of type $\phi \tau \phi^{-1}  \tau$, as in the previous lemma, are commutators in $\mathrm{GL}_2(\Z)$, though not in $\SL_2(\Z)$. 
For our purposes, the key fact is that even powers of positive Dehn twists can be expressed in this way.
\begin{cor}
	\label{cor:even_powers_are_fillable}
	For every positive natural number $k$, the torus bundle $T(B^{2k})$ has an aspherical filling with $\chi = 0$ and $\sigma = 1$.
\end{cor}
\begin{proof}
	The claim follows from the fact that $B^k \tau B^{-k} \tau = B^{2k}$, together with Lemma~\ref{lem:double_covering_trick}.
	It remains to compute the signature of the manifold $W$ obtained in Lemma~\ref{lem:double_covering_trick} when $\phi = B^k$.
	By Novikov additivity, the aspherical filling of $T(B^{2k})$ obtained from $W$ has the same signature of $W$.
	
	To compute $\sigma(W)$, we apply Wall's non-additivity formula \cite{Wal69} to the decomposition $W = W_- \cup W_+$.
	We keep the notation from the proof of Lemma~\ref{lem:double_covering_trick}, where $W_- = N(\phi)\times [0,1]$ and $W_+=N(\phi)\times [2,3]$, and the gluing locus is the manifold $X_0=N\sqcup -N$.
	We denote by $X_\pm$ the closure of $\partial W_\pm \setminus X_0$ inside $\partial W_\pm$, so that there are identifications
	\[
	X_- = (N(\phi)\times \{0\}) \sqcup (T_-\times [0,1]), \qquad X_+ = (N(\phi)\times\{2\}) \sqcup (T_+\times [0,1]),
	\]
	where $T_-$ and $T_+$ are tori.
	The common boundary
	\[
	Z\coloneqq \partial X_-=\partial X_+=\partial X_0
	\]
	is the disjoint union of two tori, which we label $Z=T_1\sqcup T_2$.
	The orientations of $W_-$ and $W_+$ induce an orientation of $W$.
	We orient $Z$, $X_-$, $X_0$ and $X_+$ so that $X_-$, $X_0$ and $X_+$ induce the same orientation on $Z$, and such that
	\[
	\partial W_- = X_0 \cup_Z (-X_-), \qquad \partial W_+ = X_+ \cup_Z (-X_0).
	\]
	
	Wall's non-additivity formula \cite{Wal69} states that
	\[
	\sigma(W)=\sigma(W_-)+\sigma(W_+)-\sigma(V,A,B,C),
	\]
	where $\sigma(V,A,B,C)$ is the corresponding Wall correction term.
	Since doubling $W_\pm$ along their boundaries yields null-cobordant closed manifolds, Novikov additivity implies that
	\[
	\sigma(W_-)=\sigma(W_+)=0.
	\]
	It remains to determine the correction term $\sigma(V,A,B,C)$.
	Let
	\[
	V = H_1(Z)=H_1(T_1)\oplus H_1(T_2)
	\]
	denote the first homology with real coefficients of $Z$. 
	This is a real vector space of dimension 4.
	We fix a basis $\{e_1,e_2,e_3,e_4\}$ of $V$ so that $H_1(T_1)=\langle e_1,e_2\rangle $ and $H_1(T_2)=\langle e_3,e_4\rangle$, where $\{e_1,e_2\}$ and $\{e_3, e_4\}$ are the bases fixed above for the torus leaf of a semibundle.
	Let $\Phi \colon V\times V \rightarrow \R$ denote the intersection form on $V$. 
	With respect to the chosen basis, we have 
	\[
	\Phi(\sum_{i=1}^4 \alpha_i e_i, \sum_{j=1}^4 \beta_j e_j) = (\alpha_1 \beta_2 - \beta_1 \alpha_2) - (\alpha_3 \beta_4 - \alpha_4 \beta_3).
	\]
	Let $A$, $B$ and $C$ denote the kernels of the maps on first homology induced by the inclusions of $Z$ into $X_-$, $X_0$ and $X_+$, respectively.
	We define
	\[
	U \coloneqq \frac{A\cap (B+C)}{(A\cap B)+(A\cap C)},
	\]
	and a bilinear form $\Psi\colon U\times U \rightarrow \R$ as follows.
	Given $[a], [a'] \in U$, represented by $a,a' \in A \cap (B+C)$, we choose $b' \in B$ and $c' \in C$ such that $a'+ b' + c' = 0$, and set
	\[
	\Psi([a],[a'])\coloneqq \Phi(a, b').
	\]
	Wall shows that $\Psi$ is indeed a well defined bilinear form on $U$ \cite{Wal69}. 
	The correction term $\sigma(V,A,B,C)$ is defined to be the signature of $\Psi$.
	
	We now determine the subspaces $A$, $B$ and $C$ in our setting.
	The coordinate $e_2$ (resp. $e_4$) is the one that is killed when mapping $T_1$ (resp. $T_2$) to the core Klein bottle of $N$ (resp. $-N$). Therefore,
	\[
	B=\langle e_2, e_4 \rangle.
	\] 
	Next, according to our attaching maps and our choices of orientations, we may assume that $T_1$ and $T_2$ are glued to $T_-\times [0,1]$ and $T_+\times [0,1]$ as described in Figure \ref{fig:signature}.
	\begin{figure}[htbp]
		\centering
		\includegraphics[scale=1]{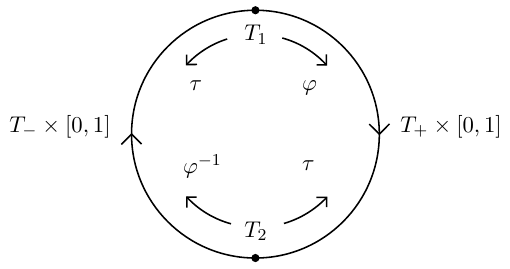}
		\caption{The boundary component $T(\phi\tau\phi^{-1}\tau)$ of $W$ can be described as follows. We glue $T_1$ to $T_-\times \{1\}$ via $\tau$, and to $T_+\times \{0\}$ via $\phi$. Likewise, we glue $T_2$ to $T_-\times \{0\}$ via $\phi^{-1}$, and to $T_+ \times \{0\}$ via $\tau$.}
		\label{fig:signature}
	\end{figure}

	It follows that, when $\phi = B^k$, the subspaces $A$ and $B$ are the kernels of the following linear maps:
	\[
	A = \ker
	\begin{pmatrix}
		-1 & 0 & 1 & -k \\
		0 & 1 & 0 & 1
	\end{pmatrix},
	\qquad
	C= \ker
	\begin{pmatrix}
		1 & k & -1 & 0 \\
		0 & 1 & 0 & 1
	\end{pmatrix}.
	\]

	Solving these systems yields:
	\[
	A = \langle e_1 + e_3, ke_1 + e_2 - e_4 \rangle, \qquad C = \langle e_1 + e_3, k e_1 - e_2 + e_4 \rangle.
	\]
	It follows that $B+ C = V$, so $A \cap (B+C)=A$. 
	Moreover, $A\cap B = \{0\}$, while $A \cap C = \langle e_1 + e_3 \rangle$.
	Thus $U$ is one dimensional, generated by the class of $ke_1 + e_2 - e_4$.
	We now compute the signature of $\Psi$. 
	Let $a = a' = ke_1 + e_2 - e_4$.
	Then $a' + b' + c' = 0$, where $b'=-2e_2 + 2e_4$ and $c' = -ke_1 + e_2 - e_4$.
	Since
	\[
	\Psi([a], [a']) = \Phi(a, b') = -2k,
	\]
	we deduce that $\sigma(V,A,B,C) = -1$, and therefore $\sigma(W)=1$.
\end{proof}
We are now ready to proceed with the proof of Theorem~\ref{thm:aspherical_manif}. 
\begin{proof}[Proof of Theorem~\ref{thm:aspherical_manif}]
	Let $n$ be a non-negative natural number.
	We want to construct a closed orientable aspherical 4-manifold $X_n$ with $\chi(X_n)=\sigma(X_n)=n$.
	Since the 4-torus has $\chi = \sigma = 0$, we may assume $n\geq 1$.
	
	Lemma \ref{lem:DiCerboStover} yields an aspherical 4-manifold $W_\alpha$ with $\chi(W_\alpha)=\sigma(W_\alpha)=1$ and $\pi_1$-injective boundary $\partial W_\alpha = T(B)\sqcup T(B^3)$. Since 
	\[
	B^{4n}=\underbrace{(B \cdot B^3)\cdots (B \cdot B^3)}_{n \text{ times}},
	\]
	by Lemma~\ref{lem:prod} there is an oriented aspherical 4-manifold $W_\beta$ with $\chi(W_\beta)=0$ and $\pi_1$-injective boundary
	\[
	\partial W_\beta = -T(B^{4n}) \sqcup (T(B)\sqcup T(B^3)) \;\sqcup\; \dots \;\sqcup\;( T(B)\sqcup T(B^3)).
	\]
	Since $4n$ is even, by Corollary~\ref{cor:even_powers_are_fillable} there is an oriented aspherical 4-manifold $W_\gamma$ with $\chi(W_\gamma)=0$, $\sigma(W_\gamma)=1$ and $\pi_1$-injective boundary
	\[
	\partial W_\gamma=T(B^{4n}).
	\]
	We obtain $X_n$ by gluing together $n$ copies of $-W_\alpha$, one copy of $W_\beta$ and one copy of $W_\gamma$ along the corresponding diffeomorphic boundary components via the identity map.
	It follows that 
	\[
	\chi(X_n)=n\cdot\chi(W_\alpha) + \chi(W_\beta) + \chi(W_\gamma) = n.
	\]
	To compute the signature of $X$, we use Novikov additivity. 
	The only piece whose signature is not yet known is $W_\beta$. 
	By the Meyer signature formula, we know that $\sigma(W_\beta)=n\cdot\mathcal{M}(B)+n\cdot\mathcal{M}(B^3)+\mathcal{M}(B^{-4n})$, where $\mathcal{M} \colon SL_2(\Z)\rightarrow \tfrac{1}{3}\Z$ denotes the Meyer function of genus one \cite[Section 4]{Mey73}.
	Meyer computed that
	\[
	\mathcal{M}(B^k)=\mathrm{sign}(k)-\frac{k}{3}
	\]
	for every $k \in \Z$ \cite[p. 23]{Mey73} (see also \cite[Section 6]{Ati87}).
	A direct substitution shows that $\sigma(W_\beta)=2n-1$.
	We conclude that
	\[
	\sigma(X_n)=-n\cdot\sigma(W_\alpha)+ \sigma(W_\beta)+\sigma(W_\gamma)=-n + 2n -1 +1 = n,
	\]
	which completes the proof of Theorem \ref{thm:aspherical_manif}.
\end{proof}

\begin{rem}
	\label{rem:singer_conjecture}
	We compute the $L^2$-Betti numbers of the manifolds $X_n$.
	We retain the notation from the proof of Theorem~\ref{thm:aspherical_manif}.
	Since $X_n$ is obtained by gluing together compact aspherical manifolds along aspherical 3-manifolds with amenable fundamental group, the additivity properties of $L^2$-Betti numbers \cite[Theorems 1.21 and 1.35]{Lue02} imply that
	\[
	b_i^{(2)}(X_n)=n\cdot b_i^{(2)}(W_\alpha) + b_i^{(2)}(W_\beta) + b_i^{(2)}(W_\gamma)
	\]
	for every $i \in \N$.
	Since the universal covering of $W_\alpha$ is a symmetric space of non-compact type, it follows from \cite[Theorem 5.12]{Lue02} that
	\[
	b_2^{(2)}(W_\alpha) = \chi(W_\alpha) = 1, \qquad b_i^{(2)}(W_\alpha) = 0, \quad \text{when } i \neq 2.
	\]
	On the other hand, $W_\beta$ is the total space of a fibration with amenable fiber, so it is $L^2$-acyclic \cite[Lemma 1.41]{Lue02}.
	Finally, the fundamental group of $W_\gamma$ can be described as the fundamental group of a graph of groups with infinite amenable vertex and edge groups, so also $W_\gamma$ is $L^2$-acyclic \cite[Theorem 7.2]{Lue02}.
	It follows that the manifolds $X_n$ satisfy the Singer conjecture.
\end{rem}

To provide positive examples for Question~\ref{quest:euler_char}, we record the following building block, whose construction is based on the \emph{torus trick} by Foozwell and Rubinstein \cite{FR16}, applied to the torus leaf of a semi-bundle.
\begin{lemma}
	\label{lem:torus_trick_semi_bundles}
	Let $\phi, \psi \in \SL_2(\Z)$.
	There is a compact oriented connected aspherical 4-manifold $W$ with $\chi(W) = 0$ and $\pi_1$-injective aspherical boundary such that $\partial W = -N(\phi) \sqcup T(\psi^{-1})\sqcup N(\psi \circ \phi)$.
\end{lemma}
\begin{proof}
	We write $N(\phi)=N\cup_\phi -N$, where $N$ and $-N$ meet along their common toric boundary $T=\partial N = \partial (-N)$. 
	Note that $T$ is incompressible in $N(\phi)$.
	
	The existence of $W$ is ensured by \cite[Lemma~2.3]{FR16}, which provides an aspherical filling for the disjoint union of any 3-manifold $M$, any 3-manifold obtained from $M$ by splitting and regluing along an incompressible surface, and a suitable torus bundle.
	We now outline Foozwell and Rubinstein's construction in the case $M=N(\phi)$.
	Observe that any manifold obtained by splitting $N(\phi)$ along $T$ and regluing the resulting boundary components via a diffeomorphism is again a semi-bundle, possibly with a different attaching map. 
	
	Let $W_1=N(\phi)\times [0,1]$, where $[0,1]$ carries the positive orientation from $\R$, and endow $W_1$ with the product orientation. The two boundary components of $W_1$ then inherit opposite orientations:
	\[
	N(\phi)\times \{0\}=-N(\phi),\qquad N(\phi)\times \{1\}=N(\phi).
	\]
	We identify a regular neighborhood $R$ of $T=\partial (-N)$ inside $-N$ with the product $T\times [0,1]=R\subseteq -N\subseteq N(\phi)$, where $T\times \{0\}$ corresponds to $\partial (-N)$.
	We pick two disjoint tori in $R\times \{1\}\subseteq W_1$ parallel to $T$, say $T_i=T\times \{i/3\}\times \{1\}$ for $i \in \{1,2\}$ (see Figure~\ref{fig:torus_trick}).
	\begin{figure}[htbp]
		\centering
		\includegraphics[scale=1]{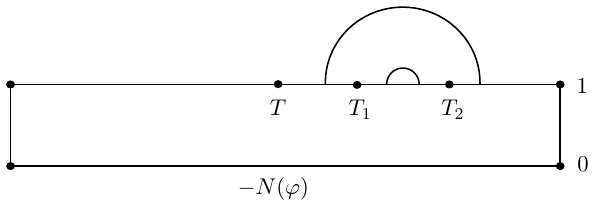}
		\caption{Gluing scheme illustrating Lemma~\ref{lem:torus_trick_semi_bundles}}
		\label{fig:torus_trick}
	\end{figure}

	Let $\epsilon$ be a positive real number such that the $\epsilon$-regular neighborhoods
	\[
	T_i(\epsilon)=T \times [i/3-\epsilon, i/3+\epsilon]\times  \{1\}
	\]
	of $T_i$ in $N(\phi)\times \{1\}$ are disjoint.
	The manifold $W$ is obtained by attaching a copy of $T\times [-\epsilon,\epsilon]\times [0,1]$ to $W_1$ so that $T\times [-\epsilon,\epsilon]\times \{0\}$ is identified with $T_1(\epsilon)$, and $T\times [-\epsilon,\epsilon]\times \{1\}$ is identified with $T_2(\epsilon)$.
	The attaching maps are chosen as follows:
	\begin{align*}
		T\times [-\varepsilon, \varepsilon]\times \{0\} \rightarrow T_1(\varepsilon), & \qquad (x,t,0)\mapsto (x, 1/3-t,1),\\
		T\times [-\varepsilon, \varepsilon]\times \{1\}\rightarrow T_2(\varepsilon), &\qquad (x,t,1)\mapsto (g(x), 2/3-t,1),
	\end{align*}
	where $g$ is any orientation-preserving diffeomorphism representing $\psi$.
	With these choices, $W$ is oriented, with boundary
	\[
	\partial W = -N(\phi) \sqcup T(\psi^{-1})\sqcup N(\psi \circ \phi).
	\]
	Moreover, by \cite[Lemma 2.3]{FR16}, $W$ is aspherical and each boundary component is $\pi_1$-injective.
\end{proof}
\begin{cor}
	\label{cor:semibundle_fillable_iff_torus_bundle}
	Let $\psi \in \SL_2(\Z)$. Then $N(\psi)$ has an aspherical filling with Euler characteristic $n \in \N$ if and only if $T(\psi)$ does.
\end{cor}
\begin{proof}
	By Lemma~\ref{lem:torus_trick_semi_bundles} with $\phi = \id$, there is a compact oriented connected aspherical 4-manifold $W$ with $\chi(W)=0$ and $\pi_1$-injective boundary $\partial W = -N(\id) \sqcup T(\psi^{-1})\sqcup N(\psi)$.
	We observe that $N(\id)$ is diffeomorphic to $T(-\id)$ \cite[Exercise 11.4.11]{Mar}.
	The manifold $T(-\id)$ is the orientable double cover of a non-orientable flat 3-manifold $M$ \cite{HW35} \cite[p.6]{CR}. 
	Hence the unique interval bundle $M \twprod I$ over $M$ with orientable total space provides an aspherical filling of $T(-\id)$ with vanishing Euler characteristic.
	It follows that the boundary component $N(\id)$ of $W$ can be capped off by an aspherical filling without changing the Euler characteristic.
	The claim then follows.
\end{proof}
Corollary~\ref{cor:semibundle_fillable_iff_torus_bundle} implies that a semi-bundle $N(\psi)$ satisfies Question~\ref{quest:euler_char} if and only if the corresponding torus bundle $T(\psi)$ does.

\section{Computing the simplicial volume}
\label{sec:computing_simplicial_volume}

In this section, we compute the simplicial volume of the building blocks introduced above and we prove Theorem~\ref{thm:virtual_small_aspherical_fillings}.

Simplicial volume is a non-negative real number $\|W,\partial W\|$ associated with every oriented compact connected topological $n$-manifold $W$ (possibly with non-empty boundary), defined as the infimum of the $\ell^1$-norms of real fundamental cycles \cite{Gro82}. 
For closed manifolds, simplicial volume is a homotopy invariant and is known to vanish whenever the fundamental group is amenable.

For manifolds with boundary, the situation is slightly more complicated. The standard construction of null-cobordisms of 3-manifolds via handle attachments from a Dehn surgery presentation -- based on the Lickorish–Wallace theorem \cite{Wal60, Lic62} -- shows that every closed 3-manifold bounds a simply connected 4-manifold. 
It follows that every 3-manifold with vanishing simplicial volume bounds a 4-manifold whose simplicial volume also vanishes \cite[Example 3.2(3)]{LMR22}.

On the other hand, since every relative fundamental cycle of a manifold induces a fundamental cycle of the boundary, one obtains the elementary inequality
\[
\|W, \partial W\| \geq \frac{\|\partial W\|}{(n+1)},
\]
which provides a simple obstruction for a 3-manifold to bound a 4-manifold with vanishing simplicial volume.
For instance, every null-cobordism of a closed hyperbolic manifold — which is known to have positive simplicial volume \cite{Gro82} — must itself have positive simplicial volume.

In the context of aspherical fillings, it is natural to restrict our attention to aspherical 3-manifolds with amenable fundamental group \cite[Section 4.2.2]{LMR22}.
Indeed, we have already seen that asphericity of fillings is preserved by gluing $\pi_1$-injective aspherical boundary components.
On the other hand, gluing manifolds along amenable $\pi_1$-injective boundaries preserves the vanishing of simplicial volume by Gromov's Additivity Theorem \cite{Gro82} \cite[Theorem 3]{BBFIPP14} \cite[Theorem 4]{Cap}. Therefore, the underlying strategy to build aspherical fillings preserves the vanishing of simplicial volume, provided that the building blocks themselves have vanishing simplicial volume.\\

A classification of closed connected 3-manifolds with amenable fundamental group is well known \cite{Sco83} \cite{AFW15}.
A list of such manifolds is given, for instance, in \cite[Corollary 2]{GGH13}. 
They are precisely the 3-manifolds that fall into one of the following classes:
\begin{enumerate}
	\item Seifert manifolds with non-negative orbifold Euler characteristic;
	\item Torus bundles;
	\item Torus semi-bundles.
\end{enumerate}
Requiring the 3-manifolds in the above list to be aspherical is equivalent to excluding the Seifert manifolds with positive orbifold Euler characteristic, since their universal cover is not contractible \cite[Theorem 5.3]{Sco83}.

The boundary components of all the building blocks considered in the previous sections are aspherical 3-manifolds with amenable fundamental group.
Moreover, every aspherical filling with vanishing Euler characteristic constructed above also has vanishing simplicial volume:
\begin{lemma}
	\label{lem:vanishing_sv}
	The manifolds from Lemma~\ref{lem:prod}, Lemma~\ref{lem:commutator}, Lemma~\ref{lem:double_covering_trick} and Lemma~\ref{lem:torus_trick_semi_bundles} all have vanishing simplicial volume.
\end{lemma}
\begin{proof}
	Let $W$ be one of the manifolds constructed in Lemmas~\ref{lem:prod} or~\ref{lem:commutator}. We know that $W$ is the total space of a fiber bundle over a compact surface with torus fiber. Hence its double $D(W)$, which is closed and is itself the total space of a fiber bundle with amenable fiber, has vanishing simplicial volume \cite{Gro82} \cite[Theorem 7.14]{Fri17}. Since the boundary of $W$ is amenable and $\pi_1$-injective, Gromov's Additivity Theorem implies $D(W)=2\cdot\|W,\partial W\|=0$, proving the claim.
	
	On the other hand, if $W$ is one of the manifolds constructed in Lemmas~\ref{lem:double_covering_trick} or~\ref{lem:torus_trick_semi_bundles}, it is straighforward from the construction to exhibit an amenable open cover of $W$ of multiplicity $2$. Hence, the simplicial volume of $W$ vanishes as a consequence of the relative vanishing theorem in presence of amenable open covers of small multiplicity \cite{Gro82} \cite[Theorem 3.13]{LMR22}.
\end{proof}

Lemma~\ref{lem:vanishing_sv} implies that every torus bundle with monodromy in $\SL_2(\Z)'$ satisfies Question~\ref{quest:sv}, i.e., it has aspherical fillings with vanishing simplicial volume. The same holds for monodromies of the form $\phi \tau \phi^{-1}\tau$ for all $\phi \in \SL_2(\Z)$, as in Lemma~\ref{lem:double_covering_trick}.
Furthermore, the simplicial volume is also preserved by the constructions underlying Corollary~\ref{cor:modulo_derived_subgroup} and Corollary~\ref{cor:semibundle_fillable_iff_torus_bundle}.
Therefore, verifying Question~\ref{quest:sv} for all torus bundles reduces to consider monodromies modulo $\SL_2(\Z)'$.
Finally, the proof of Corollary~\ref{cor:semibundle_fillable_iff_torus_bundle} shows that a semi-bundle $N(\psi)$ admits an aspherical filling with vanishing simplicial volume if and only if the corresponding torus bundle $T(\psi)$ does.

\begin{rem}
	\label{rem:sv_examples}
	The manifold $W$ from Lemma~\ref{lem:DiCerboStover}, constructed by Di Cerbo and Stover, has positive simplicial volume, since its interior carries a finite-volume complex hyperbolic metric, thus with pinched negative sectional curvature \cite{Gro82}. 
	By the proportionality principle, its simplicial volume is proportional to the Riemannian volume of its interior up to a universal proportionality constant \cite{LS09, KK15}. 
	Lemma~\ref{lem:vanishing_sv} then implies that the manifolds $X_n$ from Theorem~\ref{thm:aspherical_manif} satisfy $\|X_n\| = n \cdot \|W,\partial W\|$, which is positive for every $n\geq 1$.
\end{rem}

The problem of finding aspherical fillings with vanishing simplicial volume or Euler characteristic becomes straightforward if passing to finite covers is allowed.


\begin{proof}[Proof of Theorem~\ref{thm:virtual_small_aspherical_fillings}]
	We want to show that every oriented closed connected aspherical 3-manifold with amenable fundamental group has a finite cover which admits an aspherical filling with vanishing Euler characteristic and vanishing simplicial volume.
	Since every closed connected aspherical 3-manifold with amenable fundamental group is finitely covered by a torus bundle \cite{Sco83} \cite[Table 1]{AFW15}, it suffices to restrict our attention to torus bundles.
	
	As mentioned above, the abelianization $\SL_2(\mathbb{Z})^{\ab}$ of $\SL_2(\mathbb{Z})$ is a finite cyclic group of order 12 (see, for instance, \cite[Theorem VIII.2]{New72}).	 
	Let $T(\phi)$ denote the torus bundle with monodromy $\phi \in \SL_2(\Z)$. 
	Since the class of $\phi$ in $\SL_2(\mathbb{Z})^{\ab}$ has order at most 12, it follows that $\phi^{12}$ lies in the derived subgroup $\SL_2(\Z)'$. By Lemma~\ref{lem:commutator} and Lemma~\ref{lem:vanishing_sv}, $T(\phi^{12})$ admits an aspherical filling with vanishing simplicial volume and vanishing Euler characteristic.
	Since $T(\phi^{12})$ finitely covers $T(\phi)$, this completes the proof of Theorem~\ref{thm:virtual_small_aspherical_fillings}.
\end{proof}

\bibliography{bib}
\bibliographystyle{fram_alpha}
\end{document}